\documentclass[reqno,11pt]{amsart}

\usepackage{amsmath, amsfonts, amssymb, amsthm, amscd}
\usepackage{graphicx}
\usepackage{psfrag}
\usepackage{perpage}
\usepackage{url}
\usepackage{color}
\usepackage{soul}

\usepackage[utf8]{inputenc}
\usepackage[T1]{fontenc}
\usepackage{microtype}
\usepackage{hyperref}

\numberwithin{equation}{section}

\newtheorem{theorem}{Theorem}[section]
\newtheorem{lemma}[theorem]{Lemma}
\newtheorem{proposition}[theorem]{Proposition}

\newtheorem{definition}[theorem]{Definition}

\newcommand{\bbN}{{\ensuremath{\mathbb N}} }

\newcommand{\bbR}{{\ensuremath{\mathbb R}} }

\newcommand{\bbZ}{{\ensuremath{\mathbb Z}} }

\newcommand{\cA}{{\ensuremath{\mathcal A}} }
\newcommand{\cB}{{\ensuremath{\mathcal B}} }

\newcommand{\cE}{{\ensuremath{\mathcal E}} }
\newcommand{\cF}{{\ensuremath{\mathcal F}} }
\newcommand{\cG}{{\ensuremath{\mathcal G}} }

\renewcommand{\tilde}{\widetilde}          
\DeclareMathSymbol{\leqslant}{\mathalpha}{AMSa}{"36} 
\DeclareMathSymbol{\geqslant}{\mathalpha}{AMSa}{"3E} 
\DeclareMathSymbol{\eset}{\mathalpha}{AMSb}{"3F}     
\newcommand{\dd}{\text{\rm d}}             

\newcommand{\suptwo}[2]{\sup_{\substack{#1 \\ #2}}} 

\newcommand{\Z}{\Z}
\newcommand{\N}{\N}

\def\bs{\boldsymbol}

\newcommand{\PEfont}{\mathrm}

\DeclareMathOperator{\supess}{\ensuremath{\PEfont supess}}
\def\p{\ensuremath{\PEfont P}}
\def\e{\ensuremath{\PEfont E}}
\newcommand{\E}{\e}
\renewcommand{\P}{\p}

\newcommand\bP{\ensuremath{\bs{\mathrm{P}}}}

\newcommand{\ind}{{\sf 1}}

\renewcommand{\epsilon}{\varepsilon}
\renewcommand{\rho}{\varrho}
\renewcommand{\phi}{\varphi}
\newcommand{\tr}{\mathrm{tr}}

\newenvironment{myenumerate}{%
\renewcommand{\theenumi}{\arabic{enumi}}%
\renewcommand{\labelenumi}{{\rm(\theenumi)}}%
\begin{list}{\labelenumi}
	{%
	\setlength{\itemsep}{0.4em}%
	\setlength{\topsep}{0.5em}%
	\setlength\leftmargin{2.45em}%
	\setlength\labelwidth{2.05em}%
	\setlength{\labelsep}{0.4em}%
	\usecounter{enumi}%
	}%
	}%
{\end{list}
}

{\end{myenumerate}}

\newenvironment{myitemize}{%
\begin{list}{$\bullet$}%
 	{%
	\setlength{\itemsep}{0.4em}%
	\setlength{\topsep}{0.5em}%
	\setlength\leftmargin{2.45em}%
	\setlength\labelwidth{2.05em}%
	\setlength{\labelsep}{0.4em}%
	}%
	}%
{\end{list}}

{\end{myitemize}}

\MakePerPage[2]{footnote} 

\def\dd{\mathrm{d}}

\def\Z{{\mathbb{Z}}}
\def\N{{\mathbb{N}}}

\def\Pinv{\mathcal{P}^{\rm{inv}}}
\def\Pinverg{\mathcal{P}^{\rm{inv,erg}}}
\def\Pinvfin{\mathcal{P}^{\rm{inv,fin}}}
\def\Pinvergfin{\mathcal{P}^{\rm{inv,erg,fin}}}

\def\Iann{I^ {\rm{ann}}}
\def\Ique{I^ {\rm{que}}}
\def\tildeE{\tilde{E}}

\renewcommand{\l}{\ell}

\begin{document}

\title{Large deviation principles for words 
drawn from correlated letter sequences}

\author{F.\ den Hollander}
\address{Mathematical Institute, Leiden University, P.O.\ Box 9512,
2300 RA Leiden, The Netherlands.}
\email{denholla@math.leidenuniv.nl}

\author{J.\ Poisat}
\address{Mathematical Institute, Leiden University, P.O.\ Box 9512,
2300 RA Leiden, The Netherlands.}
\email{poisatj@math.leidenuniv.nl}

\begin{abstract}
When an i.i.d.\ sequence of letters is cut into words according to i.i.d.\ 
renewal times, an i.i.d.\ sequence of words is obtained. In the \emph{annealed} 
LDP (large deviation principle) for the empirical process of words, the rate 
function is the specific relative entropy of the observed law of words w.r.t.\ 
the reference law of words. In Birkner, Greven and den Hollander~\cite{BGdH10} 
the \emph{quenched} LDP (= conditional on a typical letter sequence) was derived 
for the case where the renewal times have an \emph{algebraic} tail. The rate 
function turned out to be a sum of two terms, one being the annealed rate function, 
the other being proportional to the specific relative entropy of the observed law 
of letters w.r.t.\ the reference law of letters, obtained by concatenating the 
words and randomising the location of the origin. The proportionality constant 
equals the tail exponent of the renewal process. 

The purpose of the present paper is to extend both LDP's to letter sequences that 
are not i.i.d. It is shown that both LDP's carry over when the letter sequence 
satisfies a mixing condition called \emph{summable variation}. The rate functions 
are again given by specific relative entropies w.r.t.\ the reference law of words, 
respectively, letters. But since neither of these reference laws is i.i.d., several 
approximation arguments are needed to obtain the extension.

\vspace{0.5cm}
{\it Acknowledgment:} FdH and JP were supported by ERC Advanced Grant 267356 VARIS.
\end{abstract}

\date{\today}

\maketitle


\section{Introduction and main results}
\label{S1}

\subsection{Notation}
\label{S1.1}

Let $E$ be a finite set of \emph{letters} and $\tilde{E} = \cup_{\ell\in\N} E^{\ell}$ the 
set of finite \emph{words} drawn from $E$. Both $E$ and $\tilde{E}$ are Polish spaces 
under the discrete topology. Write $E^\Z$ and $\tilde{E}^\Z$ for the sets of two-sided 
sequences of letters and words, endowed with the product topology, and let $\theta$ 
and $\tilde{\theta}$ denote the left-shifts acting on these sets, respectively. The set of 
probability laws on $E^\Z$ and $\tilde{E}^\Z$ that are shift-invariant, respectively, 
shift-invariant and ergodic w.r.t.\ $\theta$ and $\tilde{\theta}$ are denoted by $\Pinv(E^\Z)$ 
and $\Pinv(\tilde{E}^\Z)$, respectively, $\Pinverg(E^\Z)$ and $\Pinverg(\tilde{E}^\Z)$, 
and are endowed with the topology of weak convergence.

Let $X = (X_k)_{k\in\Z}$ be a two-sided \emph{random sequence of letters} sampled 
according to a shift-invariant probability distribution $\nu$ on $E^\Z$. Let $\tau
= (\tau_i)_{i\in\Z}$ be a two-sided i.i.d.\ sequence of \emph{renewal times} drawn 
from a common probability law $\varrho$ on $\N$, independent of $X$. The latter form 
a renewal process $T = (T_i)_{i\in\Z}$ given by
\begin{equation}
T_0=0, \qquad T_i = T_{i-1} + \tau_i, \quad i\in\Z. 
\end{equation}
Let $Y = (Y_i)_{i\in\Z}$ be the two-sided \emph{random sequence of words} cut out 
from $X$ according to $\tau$, i.e.,
\begin{equation}
\label{def:words}
Y_i = X_{(T_{i-1},T_i]} = (X_{T_{i-1}+1},\ldots,X_{T_i}), \qquad i\in\Z.
\end{equation} 
The joint law of $X$ and $\tau$ is denoted by $\P$. Write $|Y_i|$ to denote the
length of word $i$.

The reverse of cutting is glueing. The \emph{concatenation operator} $\kappa\colon\,
\tilde{E}^\Z \to E^\Z$ glues a word sequence into a letter sequence. In particular, 
$\kappa(Y) = X$. Given $Q\in\Pinv(\tilde{E}^\Z)$ with $m_Q=E_Q(|Y_1|)<\infty$, let
$\Psi_Q \in \Pinv(E^\Z)$ be defined by
\begin{equation}
\label{psiQdef}
\Psi_Q(A) = \frac{1}{m_Q} E_Q\left(\sum_{k=0}^{|Y_1|-1} 
1_{\{\theta^k\kappa(Y) \in A\}}\right), \qquad A \subset E^\Z,
\end{equation}
i.e., the law of $\kappa(Y)$ when $Y$ is drawn from $Q$, turned into a stationary 
law by randomizing the location of the origin. 

For $n\in\N$, let $(Y_{(0,n]})^{\rm{per}} \in \tilde{E}^\Z$ denote the 
$n$-periodized version of $Y$. We are interested in the \emph{empirical distribution 
of words}
\begin{equation}
\label{def:empword}
R_n = \frac{1}{n} \sum_{i=0}^{n-1} 
\delta_{\tilde{\theta}^i(Y_{(0,n]})^{\rm{per}}},
\end{equation}
both under $\P$ (= annealed law) and under $\P(\cdot \mid X)$ for $\nu$-a.a.\ $X$ 
(= quenched law). 


\subsection{Large deviation principles}
\label{S1.2}

If $\nu$ is i.i.d., then $\P$ is i.i.d.\ and the annealed LDP is standard, with the 
rate function given by the specific relative entropy of the observed law of words 
w.r.t.\ $\P$. The quenched LDP, however, is not standard. The quenched LDP was obtained 
in Birkner~\cite{B08} for the case where $\varrho$ has an exponentially bounded tail, 
and in Birkner, Greven and den Hollander~\cite{BGdH10} for the case where $\varrho$ 
has a polynomially decaying tail: 
\begin{equation}
\label{rhocond}
\lim_{ {m\to\infty} \atop {\rho(m)>0} } \frac{\log\rho(m)}{\log m} 
= -\alpha, \qquad \alpha \in [1,\infty).
\end{equation}
(No condition on the support of $\rho$ is needed other than that it is infinite.) In 
the latter case, the quenched rate function turns out to be a sum of two terms, one 
being the annealed rate function, the other being proportional to the specific relative 
entropy of the observed law of letters w.r.t.\ $\nu$, obtained by concatenating the 
words and randomising the location of the origin. The proportionality constant equals 
$\alpha-1$ times the average word length. 

The goal of the present paper is to extend both LDP's to the situation where $\nu$ is 
no longer i.i.d., but satisfies a mixing condition called \emph{summable variation},
which will be defined in Section~\ref{S3}. In what follows, $H(\cdot\mid\cdot)$ denotes
specific relative entropy (see Dembo and Zeitouni~\cite{DeZe98}, Section 6.5 for the
definition and key properties).

\begin{theorem}[{\bf Annealed LDP}]\label{annLDP}
If $\nu$ has summable variation, then the family of probability laws $\P(R_n \in \cdot\,)$, 
$n\in\N$, satisfies the LDP on $\Pinv(\tilde{E}^\Z)$ with rate $n$ and with rate function 
$\Iann\colon\,\Pinv(\tilde{E}^\Z) \mapsto [0,\infty]$ given by the specific relative entropy
\begin{equation}
\label{Iann}
\Iann(Q) = H(Q \mid \P).
\end{equation}
$\Iann$ is lower semi-continuous, has compact level sets, is affine, and has a unique zero 
at $Q=\P$.
\end{theorem}

\begin{theorem}[{\bf Quenched LDP}]\label{queLDP}
If $\nu$ has summable variation, then for $\nu$-a.a.\ $X$ the family of conditional 
probability laws $\P(R_n \in \cdot\,\mid X)$, $n\in\N$, satisfies the LDP on $\Pinv
(\tilde{E}^\Z)$ with rate $n$ and with rate function $\Ique\colon\,\Pinv(\tilde{E}^\Z) 
\mapsto [0,\infty]$ given by the sum of specific relative entropies
\begin{equation}
\label{Ique}
\Ique(Q) = H(Q \mid \P) + (\alpha-1) m_Q H(\Psi_Q \mid \nu).
\end{equation}
$\Ique$ is lower semi-continuous, has compact level sets, is affine, and has a unique zero 
at $Q=\P$.
\end{theorem}

\begin{theorem}
\label{LDPPolish}
Both LDPs remain valid when $E$ is a Polish space.   
\end{theorem}

\noindent
{\bf Remark:} 
If $m_Q=\infty$, then the second term in \eqref{Ique} is defined to be $\alpha-1$ times 
the \emph{truncation limit} $\lim_{\tr\to\infty} m_{[Q]_\tr}H(\Psi_{[Q]_\tr} \mid \nu)$, 
where $\tr$ is the operator that truncates all the words to length $\leq \tr$. Moreover, for all
$Q\in\Pinvfin(\tildeE^\bbZ)=\{\Pinv(\widetilde{E}^\Z)\colon\,m_Q<\infty\}$,
\begin{equation}\label{eq:tr.limit}
\lim_{\tr\to\infty} H([Q]_\tr \mid \P) = H(Q \mid \P),
\qquad  \lim_{\tr\to\infty} m_{[Q]_\tr} H(\Psi_{[Q]_\tr} \mid \nu) = m_Q H(\Psi_Q \mid \nu).\\
\end{equation}
See Birkner,
Greven and den Hollander~\cite{BGdH10} for details.

\medskip\noindent
{\bf Remark:}
Both rate functions are the same as for the i.i.d.\ case, even though the reference laws 
$P$ and $\nu$ are no longer i.i.d. This lack of independence will require us to go through 
several approximation arguments. Both LDP's can be applied to the problem of pinning of a 
polymer chain at an interface carrying correlated disorder. This application, which is
our main motivation for extending the LDP's, will be discussed in a future paper.


\subsection{Outline}
\label{S1.3}

In Section~\ref{S2} we collect some basic facts, introduce the relevant mixing coefficients, 
and define summable variation. We give examples where this mixing condition holds, 
respectively, fails. In Section~\ref{S3} we prove the annealed LDP by applying a result 
from Orey and Pelikan~\cite{OP88}. In Section~\ref{S4} we prove the quenched LDP by going 
over the proof in Birkner, Greven and den Hollander~\cite{BGdH10} for i.i.d.\ letter 
sequences and checking which parts have to be adapted. In Section~\ref{S5} we extend the 
LDP's from finite $E$ to Polish $E$ by using the Dawson-G\"artner projective limit LDP.


\section{Basic facts, mixing coefficients and summable variation}
\label{S2}


\subsection{Basic facts}
\label{S2.1}

Throughout the paper we abbreviate
\begin{equation}
X_{(m,n]} = (X_{m+1},\ldots,X_n), \quad Y_{(m,n]} = (Y_{m+1},\ldots,Y_n),
\qquad -\infty \leq m \leq n \leq \infty.
\end{equation}
The associated sigma-algebra's are written as
\begin{equation}
\cF_{(m,n]} = \sigma(X_{(m,n]}),
\qquad
\cG_{(m,n]} = \sigma(Y_{(m,n]}).
\end{equation}
Write $\N_0 = \N \cup \{0\}$. Let $(\nu_{x^-}(\cdot); x^- \in E^{-\N_0})$ be a 
regular version of $\nu(\cdot\mid X_{(-\infty,0]})$ (see Parthasarathy~\cite[Theorem 8.1]{P67}), 
i.e., 
\begin{equation}
\nu(A) = \int_{x^- \in E^{-\N_0}} \nu_{x^-}(A)\,\dd\nu(x^-), \qquad  
A \in \cF_{(0,\infty)}. 
\end{equation}
Since $X$ is no longer i.i.d., the distribution of a word in $Y$ depends on the outcome 
of all the previous words. However, since the word lengths are still i.i.d., when we 
condition on the past of the word sequence only the past of the letter sequence is 
relevant. This allows us to obtain a regular version of the conditional probabilities 
of $\P$ as follows.

\begin{lemma}
The collection $(\P_{y^-}(\cdot),y^-\in\tilde{E}^{-\N_0})$ of probability laws on 
$\tilde{E}^{\N}$ defined by
\begin{equation}
\label{eq:cond_prob_words}
\P_{y^-}(A) = \int_{E^\Z} \P(A \mid \cF_\Z)\,\dd\nu_{\kappa(y^-)}
\qquad \forall\, A\in \cG_{(0,\infty)},
\end{equation}
constitute a regular version of the conditional probability $\P(\cdot \mid \cG_{(-\infty,0]})$.
\end{lemma}

\begin{proof}
For every $y^-\in\tilde{E}^{-\N_0}$, $\P_{y^-}(\cdot)$ defined in \eqref{eq:cond_prob_words} 
is a probability measure. We must show that
\begin{equation}
 \int_{\tilde{E}^{-\bbN_0}} \P_{y^-}(\cdot)\,\dd\P(y^-) = \P(\cdot).
\end{equation}
By the monotone class theorem, it is enough to prove the claim for finite cylinder sets. Fix 
$r\in\N$, $(y_i)_{1\leq i \leq r} \in \tilde{E}^r$ and pick $A = \bigcap_{1 \leq i \leq r} \{Y_i=y_i\}$. 
Then
\begin{equation}
\begin{aligned}
\int_{E^\Z} \P(A \mid \cF_\Z)\,\dd\nu_{\kappa(y^-)} 
&= \int_{E^\Z} \dd\nu_{\kappa(y^-)}\,\ind_{\{X\in\kappa(A)\}}\,
\prod_{i=1}^r \rho(|y_i|)\\ 
&= \nu_{\kappa(y^-)}(X\in\kappa(A))\,\prod_{i=1}^r \rho(|y_i|),
\end{aligned}
\end{equation}
where $\kappa(A)$ is the concatenation of $A$. Since $\int_{\widetilde{E}^{-\N_0}} \dd\P(y^-)\,
\nu_{\kappa(y^-)}(\cdot) = \int_{E^{-\N_0}} \dd\nu(x^-)\,\nu_{x^-}(\cdot) = \nu(\cdot)$, we have
\begin{equation}
\int_{\tilde{E}^{-\N_0}} \dd\P(y^-)\,\int_{E^\Z} \P(A \mid \cF_\Z)\,\dd\nu_{\kappa(y^-)} 
= \nu(X\in\kappa(A))\,\prod_{i=1}^r \rho(|y_i|) = \P(A),
\end{equation}
which proves the claim.
\end{proof}


\subsection{Mixing coefficients}
\label{S2.2}

We need the following mixing coefficients for letters and words:

\begin{definition}
\label{def:mixingcoeff}
(a) For $\Lambda_1 \subset -\N_0$ and $\Lambda_2 \subset \N$, let
\begin{equation}
\label{eq:phidef}
\phi(\Lambda_1,\Lambda_2) 
= \sup_{ {x^-,\hat{x}^- \in E^{-\N_0}} \atop {(x^-)_{\Lambda_1} = (\hat{x}^-)_{\Lambda_1}} } 
\sup_{ {A\in\cF_{\Lambda_2}:} \atop {\nu_{x^-}(A)>0} } \left|\log \nu_{x^-}(A) - \log \nu_{\hat{x}^-}(A)\right|.
\end{equation}
(b) For $\Lambda \subset \N$, let
\begin{equation}
\label{eq:psidef}
\psi(\Lambda) 
= \sup_{y^-,\hat{y}^- \in \tilde{E}^{-\N_0}} 
\sup_{ {A \in \cG_{\Lambda}} \atop {P_{y^-}(A)>0} } 
\left|\log \P_{y^-}(A) - \log \P_{\hat{y}^-}(A)\right|.
\end{equation}
\end{definition}

\noindent
The restrictions $\nu_{x^-}(A)>0$ and $\P_{\hat{y}^-}(A)>0$ are put in to avoid $\infty - \infty$. Nonetheless, (\ref{eq:phidef}) and (\ref{eq:psidef}) may be infinite. Note that if $\Lambda_1 = \emptyset$, then the supremum in Definition~\ref{def:mixingcoeff}(a) 
is taken over all $x^-,\hat{x}^- \in E^{-\N_0}$ without any restriction ($(x^-)_\Lambda$ 
denotes the restriction of $x^-$ to $\Lambda$). We will use the following abbreviations:
\begin{equation}
\phi(k,\cdot) = \phi((-k,0],\cdot), \quad k\in\N, 
\qquad \phi(0,\cdot) = \phi(\emptyset,\cdot), 
\qquad \phi(\cdot,\ell) = \phi(\cdot,(0,\ell]), \quad \ell\in\N.
\end{equation}

\begin{lemma}
\label{l:fromphitopsi}
Let $0 \leq m < n$, $y_{(m,n]}\in \tilde{E}^{n-m}$ and $A = \{Y_{(m,n]} = y_{(m,n]}\}$. 
For all $y^{-},\hat{y}^- \in \tilde{E}^{-\N_0}$,
\begin{equation}
\P_{y^{-}}(A) \leq \E\left[ \exp\left\{\phi\Big(0,\Big(T_m,
T_m + \sum_{k=m+1}^n |y_k|\Big]\Big)\right\} \P_{\hat{y}^-}(A \mid T_m)\right].
\end{equation}
\end{lemma}

\begin{proof} 
Using Definition~\ref{def:mixingcoeff}(a), we have
\begin{equation}
\begin{aligned}
&\P_{y^{-}}(A) = \E\left[\nu_{\kappa(y^-)}\Big(X_{\big(T_m,T_m+\sum_{k=m+1}^n|y_k|\big]}
=\kappa (y_{(m,n]})\Big) \prod_{k=m+1}^n \rho(|y_k|)\right]\\
&\leq \E\left[ \exp\left\{\phi\left(0,\Big(T_m,T_m + \sum_{k=m+1}^n |y_k|\Big]
\right)\right\} \nu_{\kappa(\hat{y}^-)}\Big(X_{\big(T_m,T_m+\sum_{k=m+1}^n|y_k|\big]}
=\kappa(y_{(m,n]})\Big)\right.\\
&\hspace{12cm} \left.\times  \prod_{k=m+1}^n \rho(|y_k|)\right]\\
&= \E\left[ \exp\left\{\phi\left(0,\Big(T_m,T_m + \sum_{k=m+1}^n |y_k|\Big]\right)\right\} 
\P_{\hat{y}^-}(A \mid T_m)\right].
\end{aligned}
\end{equation}
\end{proof}

\begin{lemma}
\label{l:telescoping} 
For all $k\in\N_0,\ell\in\N$,
\begin{equation}
\phi(k,\ell) \leq \sum_{m=0}^{\ell-1} \phi(k+m),
\end{equation}
where $\phi(k) = \phi(k,1)$, $k\in\N_0$. 
\end{lemma}

\begin{proof}
We show that, for all $m\in\N_0$ and $k,\ell\in\N$,
\begin{equation}
\label{eq:add}
\phi(m,k+\ell) \leq \phi(m,k) + \phi(m+k,\ell),
\end{equation}
which yields the claim via iteration. To prove \eqref{eq:add}, pick $x_{(0,k+\ell]} \in 
E^{k+\ell}$ and $x^-,\hat{x}^-\in E^{-\N_0}$ with $(x^-)_{[-m,0]}=(\hat{x})_{[-m,0]}$, 
and consider the events
\begin{equation}
A_{(0,k+\ell]} = \{X_{(0,k+\ell]} = x_{(0,k+\ell]}\}, \quad
A_{(0,k]} = \{X_{(0,k]} = x_{(0,k]}\}, \quad
A_{(k,k+\ell]} = \{X_{(k,k+\ell]} = x_{(k,k+\ell]}\}.
\end{equation}
Estimate
\begin{equation}
\begin{aligned}
\nu_{x^-}(A_{(0,k+\ell]}) &= \nu_{x^-}(A_{(0,k]})\,\nu_{x^-x_{(0,k]}}(A_{(k,k+\ell]})\\
&\leq e^{\phi(m,k)}\,\nu_{\hat{x}^-}(A_{(0,k]})\,e^{\phi(m+k,\ell)}\, 
\nu_{\hat{x}^-x_{(0,k]}}(A_{(k,k+\ell]})\\
&= e^{\phi(m,k)+\phi(m+k,\ell)}\,\nu_{\hat{x}^-}(A_{(0,k+\ell]}),
\end{aligned}
\end{equation}
where $\hat{x}^-x_{(0,k]}$ is the concatenation of $\hat{x}^-$ and $x_{(0,k]}$. Insert 
this estimate into \eqref{def:mixingcoeff} and take the supremum over $x_{(0,k+\ell]}$ and 
$x^-,\hat{x}^-$ to get \eqref{eq:add}.
\end{proof}

Note that $k \mapsto \phi(k)$ is non-increasing on $\N_0$.


\subsection{Summable variation}
\label{S2.3}

The key mixing condition in our LDP's is \emph{summable variation}:
\begin{equation}
\label{eq:SV}
\mbox{(SV)} \qquad \sum_{n\in\N_0} \phi(n) < \infty. 
\end{equation}

\noindent
The term summable variation is borrowed from the theory of Gibbs measures, where
logarithms of probabilities play the role of interaction potentials, and coefficients 
similar to our $\phi(n)$'s are used to measure the absolute summability of these 
interaction potentials. 

\medskip\noindent
(I) Random processes (with finite alphabet) that satisfy (SV) include i.i.d.\  processes ($\phi(n)=0$ 
for all $n\in\N_0$), Markov chains of order $m$ ($\phi(0)<\infty$ and $\phi(n)=0$ for all $n \geq m$), 
and chains with complete connections whose one-letter forward conditional probabilities have summable 
variation. Ledrappier~\cite[Example 2, Proposition 4]{LD76} shows that such chains have a unique 
invariant measure and are Weak Bernoulli under (SV). Berbee~\cite[Theorem 1.1]{B87} shows that 
they have a unique invariant measure and are Bernoulli when $\sum_{n\in\N} \exp[-\sum_{m=1}^n \phi(m)]
=\infty$, a condition slightly weaker than (SV). (Uniqueness of the invariant measure has been proved 
more recently by Johansson and \"{O}berg~\cite{JO} and by Johansson, \"{O}berg and Pollicott~\cite{JOP} 
under the even weaker condition $\sum_{n\in\N} \phi(n)^2 <\infty$.) Yet other examples satisfying (SV) 
include Ising spins labeled by $\Z$ with a ferromagnetic pair potential that has a sufficiently thin tail
(see Berbee~\cite{B87}).

\medskip\noindent
(IIa) A class of random processes that fail to satisfy (SV) is the following. Let $E=\{0,1\}$, 
and let $p$ be any probability law on $\N$ such that $p(\ell) \sim C\ell^{-\gamma}$ for 
some $\gamma>2$. Since $\sum_{\ell\in\N} \ell p(\ell) < \infty$, there exists a stationary 
renewal process $(A_k)_{k\in\Z}$ on $\N_0$ with the following transition probabilities:
\begin{equation}
\P(A_1 = n+1 \mid A_0 = n) = \frac{\sum_{\l > n+1} p(\l)}{\sum_{\l > n} p(\l)},
\quad 
\P(A_1 = 0 \mid A_0 = n) = \frac{p(n+1)}{\sum_{\l > n} p(\l)}, \quad n\in\N_0.
\end{equation}
The process $(X_k)_{k\in\Z}$ defined by $X_k = 1_{\{A_k=0\}}$ fails to satisfy (SV). 
Indeed, pick $n\in\N$ and $x,x[n] \in E^{-\N_0}$ be such that $x_i=1$ for $i \in 
-\N_0$, $x[n]_i = 0$ for $i \in (-n,0]$ and $x[n]_i = 1$ for $i \in (-\infty,-n]$. 
Then
\begin{equation}
\phi(1) \geq \log \nu_{x}(X_1 = 1) - \log \nu_{x[n]}(X_1 = 1) 
= \log p(1) - \log \left(\frac{p(n+1)}{\sum_{\ell > n} p(\l)}\right).
\end{equation}
Since this lower bound holds for all $n\in\N$, we conclude by letting $n\to\infty$ 
that $\phi(1)=\infty$.

\medskip\noindent
(IIb) Another class of random processes that fail to satisfy (SV) is random walk in 
random scenery. Let $S=(S_n)_{n\in\Z}$ be a simple random walk on $\Z^d$, $d \geq 1$, i.e., 
$S_0=0$ and $S_n-S_{n-1}=X_n$ with $(X_n)_{n\in\Z}$ i.i.d.\ random variables uniformly 
distributed on $\{e\in \Z^ d\colon\,\|e\|=1\}$. Let $\xi = (\xi(x))_{x\in\Z^d}$ be 
i.i.d.\ random variables taking the values 0 and 1 with probability $\frac12$ each, 
and define $Z_n = (X_n,\xi(S_n))$. Then $Z=(Z_n)_{n\in\Z}$ is stationary and ergodic, 
but not i.i.d. In den Hollander and Steif~\cite[Theorems 2.4 and 2.5]{dHS97} it is 
shown that $Z$ is \emph{Weak Bernoulli} if and only if $d \geq 5$. Since (SV) implies 
Weak Bernoulli (Ledrappier~\cite[Proposition 4]{LD76}), $Z$ does not satisfy (SV) 
when $1 \leq d \leq 4$.


\section{Annealed LDP}
\label{S3}

The annealed LDP in Theorem~\ref{annLDP} is a process-level LDP. Such LDP's were proven 
by Donsker and Varadhan~\cite{DV83, DV85} for reference processes that are Markov or Gaussian. 
Orey~\cite{O84} and Orey and Pelikan~\cite{OP88} gave a proof for \emph{ratio-mixing} 
processes (see below), using the observation that any random process can be viewed
as a Markov process by keeping track of its past.

\begin{proposition}{\rm (Orey and Pelikan~\cite[Theorem 2.1]{OP88})}
\label{pr:Orey}
Suppose that $\P$ has the following ratio-mixing and continuous-dependence properties:
\begin{equation}
\label{eq:RM}
\begin{aligned}
\mbox{{\rm (RM)}} \qquad
&\hbox{There exists a non-decreasing function $n\mapsto m(n)$ such that}\\ 
&0 \leq m(n) < n, \quad \lim_{n\to\infty} m(n)/n = 0, \quad \lim_{n\to\infty} 
\psi((m(n),n])/n = 0.\\
\mbox{{\rm (CD)}} \qquad
&\hbox{For all measurable continuous functions $f\colon\,\tilde{E}^{-\bbN_0 \cup \{1\}} 
\mapsto \bbR$,}\\ 
&y^- \mapsto \int_{\tilde{E}^{-\bbN_0 \cup \{1\}}} f(y_{(-\infty,1]})\,\dd\P_{y^-}(y_{(-\infty,1]}) 
\hbox{ is continuous.}
\end{aligned}
\end{equation}
Then the family of probability laws $\P(R_n \in \cdot)$, $n\in\N$, satisfies the LDP on 
$\Pinv(\tilde{E}^{\bbZ})$ with rate $n$ and with rate function given by the specific 
relative entropy
\begin{equation}
\label{eq:specrelentr}
Q \mapsto H(Q \mid \P) = \int_{y^-\in\tilde{E}^{-\N_0}} Q(\dd y^-)
\int_{y\in \tilde{E}} Q_{y^-}|_1(\dd y)\,
\log\left(\frac{dQ_{y^-}|_1}{d\P_{y^-}|_1}(y)\right),
\end{equation}
where $Q_{y^-}|_1$ and $\P_{y^-}|_1$ are the one-word marginals of $Q_{y^-}$ and $\P_{y^-}$ 
(i.e., of $Q$ and $P$ conditional on $y^-$).
\end{proposition}

The specific relative entropy $H(Q \mid \P)$ is defined to be infinite when 
$Q_{y^-}|_1 \ll \P_{y^-}|_1$ fails on a set of $y^-$'s with a strictly positive 
${Q}$-measure. An alternative form of \eqref{eq:specrelentr} is
\begin{equation}
H(Q \mid \P) = \int_{y^-\in\tilde{E}^{-\N_0}} Q(\dd y^-)\,
h\big(Q_{y^-}(Y_1 \in \cdot\,) \mid \P_{y^-}(Y_1 \in \cdot\,)\big),
\end{equation}
where $h(\,\cdot\,|\,\cdot\,)$ denotes relative entropy. The latter can be viewed as 
the specific relative entropy of the laws of two Markov processes, namely, the laws 
of the {\it past processes} $Y^*=(Y^{(n),*})_{n\in\N}$ with $Y^{(n),*} = (Y^{(n-m)})_{m\in\N}$, 
$n\in\N$, when $Y$ is distributed according to $Q$, respectively, $P$. The regular conditional 
probability laws $(\P_{y^-}(Y_1\in\cdot\,),y^- \in \tilde{E}^{-\N_0})$ play the role of transition 
probabilities for $Y^*$, and regularity translates into the Feller property.

\medskip
We are now ready to prove Theorem~\ref{annLDP}.
\begin{proof}
Theorem~\ref{annLDP} follows by an application of Proposition~\ref{pr:Orey}, which is a rewriting of Theorem 2.1 in Orey and Pelikan~\cite{OP88}. The state
space in Orey and Pelikan~\cite{OP88} is assumed to be compact, which is not the case for $\tilde{E}$ under the discrete topology. The non-compact case is
treated by 
Orey~\cite[Theorem 5.11]{O84}. Conditions 5.8 and Eq. (3.3) in Orey~\cite{O84} correspond respectively to Conditions (RM) and (CD) in this paper. The condition in Eq. (3.2) of Orey~\cite{O84} is implied by Orey~\cite[Theorem 3.4]{O84}, which holds by
choosing the sequence of truncated state spaces $\tildeE_{(\ell_n)} = \bigcup_{1\leq k \leq \ell_n} E^k$, where $\ell_n$ is any strictly increasing sequence of integers
satisfying $\bP(T_1 > \ell_n) \leq 2^{-n}$. First we check that $\P$ satisfies (RM). From Lemma~\ref{l:fromphitopsi} and the fact that 
$\ell\mapsto\phi(0,\ell)$ is non-decreasing, we get $P_{y^-}(A) \leq e^{\phi(0,\infty)}
P_{\hat{y}^-}(A)$. Hence Definition~\ref{def:mixingcoeff}(b) gives $\psi((m,n]) \leq 
\phi(0,\infty)$ for all $0\leq m<n$. From Lemma~\ref{l:telescoping} we get
\begin{equation}
\phi(0,\infty) \leq \sum_{n\in\N_0} \phi(n).
\end{equation}
Hence, if (SV) holds, then (RM) holds for $m(n)=0$. Next we check that $\P$ satisfies 
(CD). Note that it is enough to consider measurable $f\colon\,\tilde{E} \mapsto \bbR$
because 
\begin{equation}
\P_{y^-}(Y_{(-\infty,1]}=y_{(-\infty,1]}) = \ind_{\{y_{(-\infty,0]}= y^-\}} \P_{y^-}(Y_1 = y_1).
\end{equation}
Choose $y^-$ and $\hat{y}^-$ such that $y^-_{(-k,0]} = \hat{y}^-_{(-k,0]}$ for some 
$k\in\bbN$. From Lemmas~\ref{l:fromphitopsi}--\ref{l:telescoping} we obtain
\begin{equation}
\sum_{y_1 \in \tilde{E}} f(y_1) \P_{y^-}(Y_1 = y_1) 
\leq e^{\sum_{\ell\geq k+1} \varphi(\ell)} \sum_{y_1 \in \tilde{E}} f(y_1) \P_{\hat{y}^-}(Y_1 = y_1).
\end{equation}
The same statement holds with $y^-$ and $\hat{y}^-$ interchanged. Under (SV), 
$\lim_{k\to\infty} \sum_{\ell\geq k+1} \varphi(\ell) =0$, which proves (CD).
\end{proof}


\section{Quenched LDP}
\label{S4}

In Sections~\ref{S4.1}--\ref{S4.3} we prove several lemmas that are needed in 
Section~\ref{S4.4} to give the proof of Theorem~\ref{queLDP}. This proof is an
extension of the proof in \cite{BGdH10} for i.i.d.\ $\nu$. We focus on those 
ingredients where the lack of independence of $\nu$ requires modifications.


\subsection{Decoupling inequalities}
\label{S4.1}

Abbreviate
\begin{equation}
C(\phi) = \exp\left[\sum_{n\in\N_0} \phi(n)\right] < \infty.
\end{equation}
\begin{lemma}
\label{lem:changingpast}
For all $x^-$, $\hat{x}^-\in E^{-\N_0}$, $A\in \mathcal{F}_{(0,\infty)}$ and $n\in\N$,
\begin{eqnarray}
\label{sandwich1}
&&C(\phi)^{-1} \nu_{\hat{x}^-}(A) \leq \nu_{x^-}(A) \leq C(\phi) \nu_{\hat{x}^-}(A),\\
\label{sandwich2}
&&C(\phi)^{-1} \nu_{\hat{x}^-}(A) \leq \nu\big(A \mid X_{(-n,0]}=x^-_{(-n,0]}\big) 
\leq C(\phi) \nu_{\hat{x}^-}(A).
\end{eqnarray}
\end{lemma}

\begin{proof}
To prove \eqref{sandwich1}, pick $k\in\N$ and $A\in \cF_{(0,k)}$. If $\nu_{\hat{x}^-}(A)=0$ then $\nu_{x^-}(A)=0$ as well because $\phi(k)<\infty$ and there is nothing to prove, so we can assume $\nu_{\hat{x}^-}(A)>0$. Then, by the definition
of $\phi(k)$ and Lemma \ref{l:telescoping}, 
\begin{equation}
e^{-C(\phi)} \leq e^{-\phi(0,k)} \leq \frac{\nu_{x^-}(A)}{\nu_{\hat{x}^-}(A)} 
\leq e^{\phi(0,k)} \leq e^{C(\phi)}.
\end{equation}
To prove \eqref{sandwich2}, write
\begin{equation}
\begin{aligned}
\nu\big(A \mid X_{(-n,0]}=x^-_{(-n,0]}\big)
&= \frac{\nu(\{X_{(-n,0]}=x^-_{(-n,0]}\} \cap A)}
{\nu(X_{(-n,0]}=x^-_{(-n,0]})}\\
&= \frac{\int_{\tilde{x}^-\in\E^{-\N_0}} d\nu(\tilde{x}^-)\,
\nu_{\tilde{x}^-}(\{X_{(0,n]} = x^-_{(-n,0]}\} \cap \theta^{-n}A)} 
{\int_{\tilde{x}^-\in\E^{-\N_0}} d\nu(\tilde{x}^-)\,
\nu_{\tilde{x}^-}(X_{(0,n]} = x^-_{(-n,0]})}\\
&= \frac{\int_{\tilde{x}^-\in\E^{-\N_0}} d\nu(\tilde{x}^-)\,
\nu_{\tilde{x}^-}(X_{(0,n]} = x^-_{(-n,0]}) \nu_{\tilde{x}^-x^-_{(-n,0]}}(A)}
{\int_{\tilde{x}^-\in\E^{-\N_0}} d\nu(\tilde{x}^-)\,
\nu_{\tilde{x}^-}(X_{(0,n]} = x^-_{(-n,0]})}\\
&\leq \frac{\int_{\tilde{x}^-\in\E^{-\N_0}} d\nu(\tilde{x}^-)\,
\nu_{\tilde{x}^-}(X_{(0,n]} = x^-_{(-n,0]})\,e^{C(\phi)}\,\nu_{\hat{x}^-}(A)}
{\int_{\tilde{x}^-\in\E^{-\N_0}} d\nu(\tilde{x}^-)\,
\nu_{\tilde{x}^-}(X_{(0,n]} = x^-_{(-n,0]})}\\ 
&= e^{C(\phi)}\,\nu_{\hat{x}^-}(A),
\end{aligned}
\end{equation}
where $\tilde{x}^-x^-_{(-n,0]}$ is the concatenation of $\tilde{x}^-$ and $x^-_{(-n,0]}$, 
and the inequality uses \eqref{sandwich1}. The reverse inequality is obtained 
in a similar manner.
\end{proof}

\begin{lemma}
\label{lem:decoupling} 
Let $m\in\N$, and let $(i_1,\ldots,i_m)$, $(j_1,\ldots,j_m)$ be two collections of 
integers satisfying $i_1 < j_1 \leq i_2 < j_2 \leq \ldots < i_{m-1} < j_{m-1} \leq 
i_m < j_m$. For $1\leq k\leq m$, let $A_k\in\cF_{(i_k,j_k]}$ and $p_k= \nu(A_k)$. 
Suppose that $\nu$ satisfies condition {\rm (SV)}. Then
\begin{equation}
\nu\left(\cap_{1\leq k \leq m} A_k\right) 
\leq C(\phi)^{m-1} \prod_{1\leq k \leq m} p_k. 
\end{equation}
\end{lemma}

\begin{proof} 
We give the proof for $m=2$. The general case can be handled by induction. Let 
$i_1 < j_1 \leq i_2 < j_2$, $A_1 \subset E^{j_1 - i_1}$ and $A_2 \subset E^{j_2 - i_2}$. 
For all ${x^-}\in\E^{-\N_0}$,
\begin{equation}
\begin{aligned}
&\nu\big(X_{(i_1,j_1]} \in A_1, X_{(i_2,j_2]} \in A_2\big)\\ 
&= \sum_{\substack{x_{(i_1,j_1]}\in A_1 \\ x_{(i_2,j_2]}\in A_2}} 
\nu\big(X_{(i_1,j_1]} = x_{(i_1,j_1]}, X_{(i_2,j_2]} = x_{(i_2,j_2]}\big)\\
&= \sum_{\substack{x_{(i_1,j_1]}\in A_1 \\ x_{(i_2,j_2]}\in A_2}} 
\nu\big(X_{(i_1-j_1,0]} = x_{(i_1,j_1]}, X_{(i_2-j_1,j_2-j_1]} = x_{(i_2,j_2]}\big)\\
&= \sum_{\substack{x_{(i_1,j_1]}\in A_1 \\ x_{(i_2,j_2]}\in A_2}} 
\nu\big(X_{(i_1-j_1,0]} = x_{(i_1,j_1]}\big)\, 
\nu\big(X_{(i_2-j_1,j_2-j_1]} = x_{(i_2,j_2]} \mid X_{(i_1-j_1,0]} = x_{(i_1,j_1]}\big)\\
&\leq C(\phi) \sum_{\substack{x_{(i_1,j_1]}\in A_1 \\ x_{(i_2,j_2]}\in A_2}} 
\nu\big(X_{(i_1-j_1,0]} = x_{(i_1,j_1]}\big)\, 
\nu_{x^-}\big(X_{(i_2-j_1,j_2-j_1]} = x_{(i_2,j_2]}\big)\\
&= C(\phi)p_1 \sum_{x_{(i_2,j_2]}\in A_2} 
\nu_{x^-}\big(X_{(i_2-j_1,j_2-j_1]} = x_{(i_2,j_2]}\big),
\end{aligned}
\end{equation}
where the inequality uses \eqref{sandwich2} in Lemma \ref{lem:changingpast}. Averaging 
${x^-}$ w.r.t.\ $\nu$, we get 
\begin{equation}
\nu(X_{(i_1,j_1]} \in A_1, X_{(i_2,j_2]} \in A_2) 
\leq C(\phi)p_1 p_2.
\end{equation}
\end{proof}


\subsection{Successive occurrences of patterns}
\label{S4.2}

\begin{lemma}
\label{lem:rectimes}
Fix $m \in \N$ and let  $A \in \cF_{(0,m]}$ be such that $\nu(A)>0$. Let 
$(\sigma_n)_{n\in\Z}$ be defined by
\begin{equation}
\begin{aligned}
\sigma_0 &= \inf\{k\geq 0\colon\, \theta^kX\in A\}+m,\\
\forall\, n\in\N, \qquad \sigma_n &= \inf\{k\geq\sigma_{n-1}\colon\,
\theta^k X \in A\}+m,\\
\forall\,n\in-\N, \qquad \sigma_n &= \sup\{k\leq\sigma_{n+1}-2m\colon\,
\theta^k X \in A\}+m.
\end{aligned}
\end{equation}
If $\nu$ satisfies condition {\rm (SV)}, then $\nu$-a.s.,
\begin{equation}
\label{eq:sumlogsigma}
\limsup_{n\to\infty} \frac{1}{n} \sum_{1\leq \l\leq n} 
\log[\sigma_{\l} - \sigma_{\ell-1}] 
\leq \log E_{\nu}[\sigma_1] + \log C(\phi).
\end{equation}
\end{lemma}

\begin{proof}
The strategy of proof consists in writing the sum in \eqref{eq:sumlogsigma} as an 
additive functional of an ergodic process and to use Birkhoff's ergodic theorem. First 
note that the sequence of times $(\sigma_n)_{n\in\Z}$ cuts a sequence of blocks
$B=(B_n)_{n\in\Z}$ out of the letter sequence $X$ given by
\begin{equation}
B_n = X_{(\sigma_{n-1},\sigma_n]}\in \tilde{E}.
\end{equation}
Each of these blocks belongs to the following subset of words:
\begin{equation}
\tilde{E}_{A} = \big\{y\in \tilde{E}\colon\,|y|\geq m;\,
\forall\,0\leq k<|y|-m\colon\,y_{(k,k+m]} \notin A;\,y_{(|y|-m,|y|]} \in A\big\}.
\end{equation}
Define the process $B^{\star}=(B^{\star}_n)_{n\in\Z}$ in $E^{-\N_0}$ by putting
$B^{\star}_n = X_{(-\infty,\sigma_n]}$. This process is Markovian and its transition 
kernel is given by
\begin{equation}
\P^\star_A(\hat{x}|x) = \P(B^{\star}_{n+1} = \hat{x} \mid B^{\star}_n = x) 
= \sum_{y\in \tilde{E}_A} \ind_{\{\hat{x} = xy\}} \nu_x(X_{(0,|y|]}=y), 
\qquad x,\hat{x} \in E^{-\N_0},
\end{equation}
where $xy$ is the concatenation of $x$ and $y$. For the collection $(\P^\star_A(\cdot|x),
x\in E^{-\N_0})$ to be a proper transition kernel, $\sigma_1$ must be $\nu_x$-a.s.\ finite 
for all $x\in E^{-\N_0}$. Since $\nu(A)>0$, we know from the Recurrence Theorem in 
Halmos~\cite{H60} that $\sigma_1$ is $\nu$-a.s.\ finite. But since $\nu$ and $(\nu_x)_{x\in
E^{-\N_0}}$ are equivalent under condition (SV) (note that $C(\phi)^{-1}\nu(\cdot)\leq
\nu_x(\cdot) \leq C(\phi)\nu(\cdot)$ as a consequence of \eqref{sandwich1} in 
Lemma~\ref{lem:changingpast}), $\sigma_1$ indeed is $\nu_x$-a.s.\ finite for all 
$x\in E^{-\N_0}$. Since (with a slight abuse of notation) the $B^{\star}_n$'s are also 
in $E^{-\N_0} \times \tilde{E}_A$, we can write 
\begin{equation}
\sum_{1\leq \ell\leq n} \log[\sigma_\ell - \sigma_{\ell-1}] 
= \sum_{1\leq \ell\leq n} \log |\pi(B^{\star}_{\ell})|,
\end{equation}
where $\pi$ is defined by $\pi\colon\,(u,v)\in E^{-\N_0}\times \tilde{E}_A \mapsto v$. 
We next apply Birkhoff's ergodic theorem to the sum in the right-hand side, i.e., to 
the process $B^{\star}$. This process has a stationary distribution, which we denote 
by $\P^\star_A$. It is easy to check that $\P^\star_A$ is the law of $X_{(-\infty,\sigma_0]}$ 
conditional on the event $\cap_{\ell\in -\N_0}\{\sigma_{\ell} > -\infty\}$, which has 
probability one according to the Recurrence Theorem. Again using \eqref{sandwich1}
in Lemma~\ref{lem:changingpast}, we see that for all sets $\cA$ and $\cB$ that
are measurable w.r.t.\ $\sigma(B_{(-\infty,0]})$ and $\sigma(B_{(0,\infty)})$,
respectively, 
\begin{equation}
C(\phi)^{-1} \P_A(\cA)\P_A(\cB) \leq \P_A(\cA \cap \cB) \leq C(\phi) \P_A(\cA)\P_A(\cB),
\end{equation} 
where $\P_A$ is the law of $B$ induced by $\P^\star_A$. Therefore $\P_A$ is Weak 
Bernoulli (Ledrappier~\cite{LD76}), and hence is ergodic. Thus, we have
\begin{equation}
\lim_{n\to\infty} \frac{1}{n}\sum_{1\leq \l\leq n} 
\log[\sigma_\l - \sigma_{\l-1}] = E_{\P_A}(\log [\sigma_1 - \sigma_0])\, \leq \log E_{\P_A}(\sigma_1 - \sigma_0).
\end{equation}
Moreover, for all $\hat{x}^-\in E^{-\N_0}$,
\begin{equation}
E_{\P_A}(\sigma_1 - \sigma_0) = \int E_{\nu_{x^-}}(\sigma_1 - \sigma_0) \dd \P_A(x^-)  
\leq C(\phi) E_{\nu_{\hat{x}^-}} (\sigma_1 - \sigma_0),
\end{equation}
which gives $E_{\P_A}(\sigma_1 - \sigma_0) \leq C(\phi) E_\nu(\sigma_1 - \sigma_0)$ 
and completes the proof.
\end{proof}


\subsection{Decomposition of relative entropy}
\label{S4.3}

Write $H(Q)$ to denote the specific entropy of $Q$. Let 
\begin{equation}
\begin{aligned}
\Pinvfin(\widetilde{E}^\Z)
&= \{\Pinv(\widetilde{E}^\Z)\colon\,m_Q<\infty\},\\
\Pinvergfin(\widetilde{E}^\Z)
&= \{\Pinv(\widetilde{E}^\Z)\colon\,\text{Q is ergodic},\, m_Q<\infty\}.
\end{aligned}
\end{equation}

\begin{lemma}
\label{l:decEntropy}
Suppose that $\phi(0)<\infty$. Then, for all $Q\in \Pinvfin(\widetilde{E}^\Z)$,
\begin{equation}
\label{hreldecomp}
\begin{aligned}
H(Q \mid \P) &= -H(Q) - E_Q[\log \rho(\tau_1)] 
- m_Q E_{\Psi_Q}[\log \nu_{X_{(-\infty,0]}}(X_1)],\\
H(\Psi_Q \mid \nu) &= -H(\Psi_Q) - E_{\Psi_Q}[\log \nu_{X_{(-\infty,0]}}(X_1)].
\end{aligned}
\end{equation}
\end{lemma}

\begin{proof}
To get the first relation, write $H(Q \mid \P) = -H(Q) - \E_Q[\log \P_{Y_{(-\infty,0]}}(Y_1)]$,
\begin{equation}
\E_Q[\log \P_{Y_{(-\infty,0]}}(Y_1)] 
= \E_Q[\log \rho(\tau_1)] + \E_Q[\log \nu_{X_{(-\infty,0]}}(X_{(0,\tau_1]})]
\end{equation}
and (recall \eqref{psiQdef})
\begin{equation}
\label{eq:aux}
\E_Q[\log \nu_{X_{(-\infty,0]}}(X_{(0,\tau_1]})] 
= \E_Q\left[\sum_{k=0}^{\tau_1-1} \log \nu_{X_{(-\infty,k]}}(X_{k+1})\right] 
= m_Q \E_{\psi_Q} [\log \nu_{X_{(-\infty,0]}}(X_1)],
\end{equation}
where we use the abbreviation $\nu_{x^{-}}(x_\Lambda) = \nu_{x^{-}}(X_\Lambda=x_\Lambda)$,
$\Lambda \subset \N$. The second relation follows in a similar manner.
\end{proof}

\noindent
All terms in the right-hand side of \eqref{hreldecomp}, except possibly $H(Q)$, are finite
because $E$ is finite, $\rho$ satisfies \eqref{rhocond}, and $\phi(0)<\infty$. 

\begin{lemma}
\label{lem:asymptEntropy}
If $\nu$ satisfies condition {\rm (SV)}, then for all $Q\in\Pinvergfin(\widetilde{E}^\N)$,
\begin{equation}
\lim_{n\to\infty} \frac{1}{n} \log \nu(X_{(0,T_n]}) 
= m_Q E_{\Psi_Q}[\log \nu_{X_{(-\infty,0]}}(X_1)] \quad Q-a.s.
\end{equation}
\end{lemma}

\begin{proof}
First observe that \eqref{sandwich2} in Lemma \ref{lem:changingpast} gives
\begin{equation}
\label{lognu0}
C(\phi)^{-1} \nu_{X_{(-\infty,0]}}(X_{(0,T_n]}) \leq \nu(X_{(0,T_n]}) 
\leq  C(\phi) \nu_{X_{(-\infty,0]}}(X_{(0,T_n]}).
\end{equation}
Next write
\begin{equation}
\label{lognu1}
\log \nu_{X_{(-\infty,0]}}(X_{(0,T_n]}) 
= \sum_{k=0}^{T_n-1} \log \nu_{X_{(-\infty,k]}}(X_{k+1}) 
= \sum_{i=0}^{n-1} \sum_{k=T_i}^{T_{i+1}-1} \log \nu_{X_{(-\infty,k]}}(X_{k+1}).
\end{equation}
Use (\ref{lognu1}) and the ergodicity of $Q$ to obtain, for $Q$-a.s. $Y$,
\begin{equation}
\label{lognu3}
\lim_{n\to\infty} \frac{1}{n}\log \nu_{X_{(-\infty,0]}}(X_{(0,T_n]}) 
= \E_Q \left[ \sum_{k = 0}^{\tau_1 -1} \log \nu_{X_{(-\infty,k]}}(X_{k+1}) \right] 
= m_Q \E_{\Psi_Q} [\log \nu_{X_{(-\infty,0]}}(X_{1})].
\end{equation}
Combine (\ref{lognu0}--\ref{lognu3}) to get the claim.
\end{proof}


\subsection{Proof of quenched LDP}
\label{S4.4}  

We are now ready to give the proof of Theorem~\ref{queLDP}. 

\begin{proof}
The proof is an extension of the proof in \cite{BGdH10} for i.i.d.\ $\nu$. Since 
the latter is rather long, it is not possible to repeat all the ingredients here. 
Below we restrict ourselves to indicating the necessary \emph{modifications}, which 
are based on the results in Sections~\ref{S4.1}--\ref{S4.3}. We leave it to the 
reader to go over the full proof in \cite{BGdH10} and check that, indeed, these 
are the only modifications needed.

\medskip\noindent
\emph{Decomposition of relative entropies.} 
Replace \cite[Eqs.(1.25--1.26)]{BGdH10} by the relations in Lemma~\ref{l:decEntropy}. 
These relations allow us to decompose $I^{\rm que}$ as a sum of three terms that
appear in the proofs of the lower bound and the upper bound of the LDP as given in 
\cite[Section 1.3]{BGdH10}.

\medskip\noindent
\emph{Upper bound.}
The upper bound in \cite[Proposition 3.1]{BGdH10} is proved by first restricting to 
$Q\in \Pinvergfin(\tilde{E}^\Z)$. The event in \cite[Eq. (3.4)]{BGdH10} is used to 
define a suitable neighbourhood of $Q$. In that equation only the fourth line has 
to be replaced by
\begin{equation}
\left\{ \frac{1}{M} \log \nu({X_{(0,T_M]}}) \in 
m_Q E_{\Psi_Q}\big[\log \nu_{X_{(-\infty ,0]}}(X_1)\big] 
+ [-\epsilon_1,\epsilon_1]  \right\}.
\end{equation}
By Lemma~\ref{lem:asymptEntropy}, the intersection event in \cite[Eq. (3.4)]{BGdH10}
still has probability at least $1-\delta_1/4$ for $M$ large enough. Also
\cite[Sections 3.2--3.3]{BGdH10} are unchanged. The next (harmless) modification is in 
\cite[Eq.(3.39)]{BGdH10}, which has to be replaced by
\begin{equation}\label{eq:ineq_Ak}
\P\left(\cap_{1 \leq k \leq n} \{A_k = a_k\}\right) 
\leq [C(\phi)p]^{\sum_{1 \leq k \leq n} a_k},
\end{equation}
where $A_k$ is the indicator defined in \cite[Eqs.(3.36--3.37)]{BGdH10}, and $a_k \in 
\{0,1\}$ labels whether or not at some specific location of the letter sequence $X$ there 
is a string of letters arising from the concatenation of $Q${\it -typical} words (see 
\cite[Eq (3.5--3.6)]{BGdH10}). The inequality in (\ref{eq:ineq_Ak}) is proved via 
Lemma~\ref{lem:decoupling} and allows us to use \cite[Lemma 2.1]{BGdH10}, which 
controls the occurrence of certain patterns in $X$. We are then able to complete 
the argument in \cite[Section 3.4]{BGdH10}. 

A further step consists in removing the ergodicity assumption on $Q$. The argument 
in \cite[Section 3.5]{BGdH10} is long and technical, but carries over essentially verbatim
because Lemmas~\ref{lem:changingpast}--\ref{lem:decoupling} allow us, for arbitrary 
cylinders events, to replace $\nu$ by the product of its one-letter marginals at the expense 
of a finite factor.

\medskip\noindent
\emph{Lower Bound.} 
The lower bound in \cite[Proposition 4.1]{BGdH10} is proved by bounding from below 
the probability that $R_n$ lies in a neighbourhood of some $Q\in \Pinvfin(\tilde{E}^\Z)$. 
When $Q$ is ergodic we can use the same strategy as in \cite{BGdH10} (namely, by 
jumping to $Q${\it -typical} substrings of letters), but a modification is needed to go 
from \cite[Eq.(4.7)]{BGdH10} to \cite[Eq.(4.8)]{BGdH10}, since the increments of the 
$\sigma_{\ell}^{(M)}$, $\ell\in\N$, defined in \cite[Eq.(4.6)]{BGdH10} are no longer i.i.d. 
This can again be handled with the help of Lemma~\ref{lem:rectimes}. Note that the 
extra constant $\log C(\phi)$ is killed by letting $M \to \infty$ in \cite[Eq. (3.8)]{BGdH10}. 
Using ergodic decomposition, we get rid of the ergodicity assumption on $Q$, 
exactly as in \cite[Eqs. (4.9--4.11)]{BGdH10}.
\end{proof}

\medskip\noindent
\emph{Truncation limits.}
The argument in \cite[Section 3]{BGdH10} also uses \cite[Lemma A.1]{BGdH10},
which in our case is (\ref{eq:tr.limit}).
The proof in \cite[Appendix A]{BGdH10} carries over verbatim, with obvious modifications 
in \cite[Eqs. (A.3--A.4) and (A.13--A.14)]{BGdH10}. 


\section{Extension to Polish spaces}
\label{S5}

In this section we prove Theorem~\ref{LDPPolish}, i.e., we extend the LDP's in
Theorems~\ref{annLDP}--\ref{queLDP} from a finite letter space to a Polish letter 
space. We first prove the LDP's for a sequence of \emph{coarse-grained} finite 
letter spaces associated with a sequence of nested finite partitions of the Polish 
letter space. After that we apply the \emph{Dawson-G\"artner projective limit LDP} 
(see Dembo and Zeitouni~\cite{DeZe98}, Lemma 4.6.1). A somewhat delicate point is 
that (SV) for the full process does not necessarily imply (SV) for the coarse-grained 
process. Indeed, the first supremum in \eqref{eq:phidef} decreases under 
coarse-graining while the second supremum increases. The way out is to use (SV)
for the full process to prove the decoupling inequalities in Section~\ref{S4.1} 
for the coarse-grained process.

Let $X=(X_k)_{k\in\Z}$ be a stationary process on a Polish space $(E,d)$, with 
$(\nu_{x^-}(\cdot)$, $x^-\in E^{-\N_0})$ a regular version of the conditional 
probability $\nu(\cdot \mid X_{(-\infty, 0]})$ satisfying condition (SV), i.e.,
\begin{equation}
\label{cdtSV}
C(\phi) = \exp\left[\sum_{n\in\N_0} \phi(n)\right] < \infty,
\end{equation}
where
\begin{equation}
\phi(n) = \sup_{\substack{x^-,\hat{x}^- \in E^{-\N_0}\colon \\ 
d(x^-,\hat{x}^-)\leq 2^{-n}}} 
\sup_{\substack{A\in\cF_1\colon \\ \nu_{x^-}(A)>0}} 
|\log \nu_{x^-}(A) - \log \nu_{\hat{x}^-}(A)|
\end{equation}
with
\begin{equation}
d(x^-, \hat{x}^-) = \sum_{k\in\N_0} 2^{-(k+1)}\,
\big[1 \wedge d\big(x^-_{-k},\hat{x}^-_{-k}\big)\big].
\end{equation}
We assume that, for any $x^-,\hat{x}^- \in E^{-\N_0}$, the measures $\nu_{x^-}|_1 
= \nu_{x^-}(X_1 \in \cdot\,)$ and $\nu_{\hat{x}^-}|_1 = \nu_{\hat{x}^-}(X_1 \in \cdot\,)$ 
are equivalent, so that the Radon-Nikodym derivative $\dd \nu_{x^-}|_1/\dd 
\nu_{\hat{x}^-}|_1$ exists and
\begin{equation}
\sup_{\substack{A\in\cF_1\colon \\ \nu_{x^-}(A)>0}} 
\big[ \log \nu_{x^-}(A) - \log \nu_{\hat{x}^-}(A) \Big] 
= \supess \Big[ \log \frac{\dd \nu_{x^-}|_1}{\dd \nu_{\hat{x}^-}|_1}\Big],
\end{equation}
leading to the alternative definition
\begin{equation}
\phi(n) = \sup_{\substack{x^-, \hat{x}^- \in E^{-\N_0}\colon \\ d(x^-, \hat{x}^-)\leq 2^{-n}}} 
\supess \Big[ \log \frac{\dd \nu_{x^-}|_1}{\dd \nu_{\hat{x}^-}|_1} \Big].
\end{equation}
Similarly as in Section~\ref{S2.3}, we note that (SV) holds for i.i.d.\ processes, for
Markov chains of finite order with $\phi(0)<\infty$, and a subclass of chains with 
complete connections whose letter space is countable (Berbee~\cite{B87}). Other 
examples are rotators that are labelled by $\Z$, take values in the unit circle, and 
interact with each other according to a Hamiltonian with long-range potentials that 
have a sufficiently thin tail, as can be easily checked by hand.  

\medskip
The following lemma generalizes \eqref{sandwich1} in Lemma~\ref{lem:changingpast}.

\begin{lemma}
\label{lem:changingpastPolish1}
For all $x^-$,$\hat{x}^-\in E^{-\N_0}$ and $A\in\cF_{(0,\infty)}$,
\begin{equation}
C(\phi)^{-1} \nu_{\hat{x}^-}(A) \leq \nu_{x^-}(A) \leq C(\phi) \nu_{\hat{x}^-}(A).
\end{equation}
\end{lemma}
\begin{proof}
For all $x^-,\hat{x}^-\in E^{-\N_0}$ and $n\in\N$,
\begin{align}
\frac{\dd \nu_{x^-}|_n}{\dd \nu_{\hat{x}^-}|_n}(x_1,\ldots, x_n) 
&= \frac{\dd \nu_{x^-}|_1}{\dd \nu_{\hat{x}^-}|_1}(x_1)
\times \frac{\dd \nu_{x^-x_1}|_1}{\dd \nu_{\hat{x}^-x_1}|_1}(x_2) 
\times \dots \times
\frac{\dd \nu_{x^-x_1 \cdots x_{n-1}}|_1}{\dd \nu_{\hat{x}^-x_1 \cdots x_{n-1}}|_1}(x_n)\\
&\leq \exp[\phi(0) + \phi(1) + \cdots + \phi(n-1)] \leq C(\phi),\nonumber
\end{align}
where $\nu_{x^-}|_n$ denotes the $n$-letter marginal conditional on $x^-$. This proves the claim.
\end{proof}

Let $\cE_c = \{E_1,\ldots,E_c\}$, $c\in\N$, be a finite partition of $E$. Identify 
$\cE_c^\Z$ with $\{1,\ldots,c\}^\Z$. Let $X^{(c)}=(X_k^{(c)})_{k\in\Z}$ on $\cE_c^\Z$ 
be the coarse-graining of $X$ on $E^\Z$ defined by
\begin{equation}
\label{eq:defcg}
X_n^{(c)} = \sum_{i=1}^c i\,\ind_{\{X_n \in E_i\}}.
\end{equation}
Write $\cF^{(c)}_{(0,\infty)}=\sigma(X^{(c)}_{(0,\infty)})$. The following lemma generalizes 
\eqref{sandwich2} in Lemma~\ref{lem:changingpast}.

\begin{lemma}
\label{lem:changingpastPolish2}
For all $x^- \in E^{-\N_0}$, $c \in \N$, $i^-,j^- \in \{1,\ldots,c\}^{-\N_0}$, $A \in 
\cF^{(c)}_{(0,\infty)}$ and $m,n\in\N$,
\begin{eqnarray}
&&C(\phi)^{-1} \nu_{x^-}(A) 
\leq \nu\left(A \mid X^{(c)}_{(-n,0]} = i^-_{(-n,0]}\right) 
\leq C(\phi) \nu_{x^-}(A),
\label{eq:cg1}\\[0.2cm]
&&C(\phi)^{-2} \nu\left(A \mid X^{(c)}_{(-m,0]} = j^-_{(-m,0]}\right) 
\leq \nu\big(A\mid X^{(c)}_{(-n,0]} = i^-_{(-n,0]}\big) \nonumber\\ 
&&\qquad\qquad\qquad\qquad\qquad\qquad\qquad\qquad 
\leq C(\phi)^2 \nu\left(A \mid X^{(c)}_{(-m,0]} = j^-_{(-m,0]}\right),
\label{eq:cg2}\\[0.2cm]
&&C(\phi)^{-1} \nu\left(A \mid X^{(c)}_{(-n,0]} = i^-_{(-n,0]}\right) 
\leq \nu(A) \leq C(\phi) \nu\left(A \mid X^{(c)}_{(-n,0]} = i^-_{(-n,0]}\right),
\label{eq:cg3}
\end{eqnarray}
provided that the events on which we condition have positive probability.
\end{lemma}

\begin{proof}
Note that \eqref{eq:cg2} follows by applying \eqref{eq:cg1} twice, while \eqref{eq:cg3} 
follows by integrating $x^-$ w.r.t.\ $\nu$ in \eqref{eq:cg1}. Therefore it suffices to 
prove \eqref{eq:cg1}. To that end write
\begin{equation}
\nu\left(A\mid X^{(c)}_{(-n,0]}=i^-_{(-n,0]}\right)
= \frac{\int_{E^{-\N_0}} \nu_{\tilde{x}^-}(\{X^{(c)}_{(0,n]}=i^-_{(-n,0]}\}\cap \theta^{-n}A)\,
\dd\nu(\tilde{x}^-)}
{\nu(X^{(c)}_{(-n,0]}=i^-_{(-n,0]})}.
\end{equation}
The integral in the numerator equals
\begin{equation}
\int_{E^{-\N_0}} \Big[ \int_{E^n} \dd\nu_{\tilde{x}^-}(x_{(0,n]})\,
\ind_{\{x^{(c)}_{(0,n]} = i^-_{(-n,0]}\}}\,
\nu_{\tilde{x}^-,x_{(0,n]}}(A) \Big]\,\dd\nu(\tilde{x}^-),
\end{equation}
from which the claim follows via Lemma~\ref{lem:changingpastPolish1}.
\end{proof}

The following lemma is another consequence of (SV).

\begin{lemma}
\label{lem:spec_rel_ent_superadd}
Under condition {\rm (SV)},
\begin{equation}
H(Q \mid \P) = \sup_{n\in\N} \frac{1}{n} \Big\{ h\big(Q(Y_{(0,n]} \in\cdot) \mid \P(Y_{(0,n]} \in\cdot)\big)
- \log C(\phi) \Big\},
\end{equation}
and the supremum is also a limit. The same result holds when $(\P,Q)$ is replaced by $(\P^{(c)},Q^{(c)})$
or $(\nu,\Psi_Q)$ or $(\nu^{(c)},\Psi_Q^{(c)})$.
\end{lemma}

\begin{proof}
We prove the result for $(\P,Q)$. The other cases are similar. For $n\in\N$, let $\cB(\tilde{E}^n)$ 
be the set of bounded measurable functions on $\tilde{E}^n$. From the variational characterization of 
relative entropy (see Dembo and Zeitouni~\cite[Lemma 6.2.13]{DeZe98}), we get that for all $n,m\in\N$,
\begin{align}
&h\big(Q(Y_{(0,n+m]} \in\cdot)\mid \P(Y_{(0,n+m]} \in\cdot)\big) \\
&= \sup_{f\in \cB(\tilde{E}^{n+m})} \Big\{ E_Q[f(Y_{(0,n+m]})] 
- \log E_{\P}\Big[e^{f(Y_{(0,n+m]})} \Big] \Big\}\nonumber\\
& \geq \suptwo{f_1 \in \cB(\tilde{E}^n)}{f_2 \in \cB(\tilde{E}^m)} 
\Big\{ E_Q\big[f_1(Y_{(0,n]})+f_2(Y_{(n,n+m]})\big] - \log E_{\P}\Big[e^{f_1(Y_{(0,n]})}
e^{f_2(Y_{(n,n+m]})} \Big] \Big\}.\nonumber
\end{align}
Using the decoupling inequality of Lemma \ref{lem:changingpastPolish1} and the stationarity of $\P$ 
and $Q$, we may bound the right-hand side from below by
\begin{align}
&\sup_{f_1\in \cB(\tilde{E}^{n})} \Big\{ E_Q[f(Y_{(0,n]})] - \log E_{\P}\Big[e^{f(Y_{(0,n]})} \Big] \Big\}\\
&\qquad + \sup_{f_2\in \cB(\tilde{E}^{m})} \Big\{ E_Q[f(Y_{(0,m]})] 
- \log E_{\P}\Big[e^{f(Y_{(0,m]})} \Big] \Big\} -\log C(\phi) \nonumber\\
&= h\big(Q(Y_{(0,n]} \in\cdot) \mid \P(Y_{(0,n]} \in\cdot)\big) 
+ h\big(Q(Y_{(0,m]} \in\cdot) \mid \P(Y_{(0,m]} \in\cdot)\big) -\log C(\phi). \nonumber
\end{align}
The claim now follows from the superadditivity of the sequence $\{h(Q(Y_{(0,n]} \in\cdot) \mid 
\P(Y_{(0,n]} \in \cdot)) - \log C(\phi)\}_{n\in\N}$.
\end{proof}

In what follows we need the notion of \emph{conditional local absolute continuity} (which is 
weaker than absolute continuity). 

\begin{definition}
Let $F$ be a finite space equipped with the discrete topology and discrete $\sigma$-algebra, and 
let $\lambda,\mu$ be two stationary probability measures on $F^{\Z}$ with respective regular 
conditional probabilities $(\lambda_{x^-},\,x^-\in F^{-\N_0})$ and $(\mu_{x^-},\, x^-\in F^{-\N_0})$. 
The law $\lambda$ is said to be conditionally locally absolutely continuous w.r.t.\ to the law 
$\mu$ \rm{(}written as $\lambda\,\,{\ll}_{\rm{cond}}\,\,\mu$\rm{)} when, for $\lambda$-a.a.\ 
$x^-$ and all $n\in\N$, $\lambda_{x^-}|_n$ is absolutely continuous w.r.t.\ to $\mu_{x^-}|_n$ 
\rm{(}written as $\lambda_{x^-}|_n \ll \mu_{x^-}|_n$\rm{)}, where $\lambda_{x^-}|_n$ and 
$\mu_{x^-}|_n$ are the marginal laws on the first $n$ coordinates.
\end{definition}

\noindent
Note that the set $\{x^-\in F^{-\N_0}\colon\, \lambda_{x^-}|_n \ll \mu_{x^-}|_n\}$ 
is measurable because it can be written as
\begin{equation}
\bigcap_{A \in \{0,1\}^{F^n}} \{\mu_{x^-}(A)>0\} \cup (\{\mu_{x^-}(A)=0\} \cap \{\lambda_{x^-}(A)=0\}).
\end{equation}

\begin{proof} 
We need to prove both the annealed LDP and the quenched LDP.

\medskip\noindent
{\bf Annealed LDP.} Lemma~\ref{lem:changingpastPolish1} shows that under condition (SV)
Lemmas~\ref{l:fromphitopsi}--\ref{l:telescoping} carry over from finite letters 
to Polish letters. Condition (CD) in Proposition~\ref{pr:Orey} is also implied by condition (SV), by an argument similar to the case of finite letters.
Therefore the ratio-mixing and continuous dependence properties of Orey and Pelikan~\cite{OP88}
again yields the annealed LDP.

\medskip\noindent
{\bf Quenched LDP.} The proof comes in 4 steps. 

\medskip\noindent
{\bf 1.} 
We first use Lemmas~\ref{lem:changingpastPolish1}--\ref{lem:changingpastPolish2} to 
show that Lemmas~\ref{lem:decoupling}--\ref{lem:asymptEntropy} carry over to the 
coarse-grained process $X^{(c)}$ defined in \eqref{eq:defcg} for every $c\in\N$. 
This is straightforward, except that Lemma~\ref{lem:asymptEntropy} carries over 
to $Q\in\Pinverg((\widetilde{\cE_c})^\Z)$ only when $\Psi_Q\,\,{\ll}_{\rm{cond}}\,\,
\nu^{(c)}$, where $\nu^{(c)}$ denotes the law of $X^{(c)}$. We will see in Step 4 
below that, because $H(\Psi_Q \mid \nu^{(c)})=\infty$ when $\Psi_Q\,\,{\ll}_{\rm{cond}}
\,\,\nu^{(c)}$ fails, this restriction does not affect the LDP.

\medskip\noindent
{\bf 2.}
To prove the restricted version of Lemma~\ref{lem:asymptEntropy}, 
let $Q\in\Pinverg((\widetilde{\cE_c})^\Z)$ be such that $\Psi_Q\,\,{\ll}_{\rm{cond}}\,\,
\nu^{(c)}$. Using the notation introduced 
below \eqref{eq:aux}, we know from Lemma~\ref{lem:changingpastPolish2} (by letting 
$n\to \infty$ in Eq.~(\ref{eq:cg3}) and using the Martingale Convergence Theorem) that, 
for $\nu^{(c)}$-a.a.\ $X^{(c)}_{(-\infty,0]}$,
\begin{equation}
\nu\Big(C(\phi)^{-1} \nu_{X^{(c)}(-\infty,0]}\big(X^{(c)}_{(0,n)}\big) 
\leq \nu\big(X^{(c)}_{(0,n)}\big) 
\leq C(\phi) \nu_{X^{(c)}(-\infty,0]}\big(X^{(c)}_{(0,n)}\big) ~\Big|~ 
X^{(c)}_{(-\infty,0]}\Big) = 1.
\end{equation}
By conditional local absolute continuity we have, for $\Psi_Q$-a.a.\ $X^{(c)}_{(-\infty,0]}$,
\begin{equation}
\Psi_Q\Big(C(\phi)^{-1} \nu_{X^{(c)}(-\infty,0]}\big(X^{(c)}_{(0,n)}\big) 
\leq \nu\big(X^{(c)}_{(0,n)}\big) 
\leq C(\phi) \nu_{X^{(c)}(-\infty,0]}\big(X^{(c)}_{(0,n)}\big) 
~\Big|~ X^{(c)}_{(-\infty,0]}\Big) = 1.
\end{equation}
This implies that, for $\Psi_Q$-a.a.\ $X^{(c)}_{(-\infty,0]}$,
\begin{equation}
C(\phi)^{-1} \nu_{X^{(c)}(-\infty,0]}\big(X^{(c)}_{(0,n)}\big) 
\leq \nu\big(X^{(c)}_{(0,n)}\big) 
\leq C(\phi) \nu_{X^{(c)}(-\infty,0]}\big(X^{(c)}_{(0,n)}\big),
\end{equation}
which settles the restricted version of Lemma~\ref{lem:asymptEntropy}.

\medskip\noindent
{\bf 3.}
By the same argument as in Section~\ref{S4.4}, we now know that the quenched LDP
holds for $X^{(c)}$ for all $c\in\N$ (see Step 4 below for comments). Picking for 
$\cE_c = \{E_1,\ldots,E_c\}$, $c\in\N$, a \emph{nested sequence} of finite partitions 
of $E$ as in \cite[Section 8]{BGdH10}, we conclude from the Dawson-G\"artner projective 
limit LDP that the quenched LDP also holds for $X$, with rate function
\begin{equation}
I^\mathrm{que}(Q) = \sup_{c\in\N} I_c^\mathrm{que}(Q^{(c)}),\quad Q\in \Pinv(\tilde{E}^{\bbZ}),
\end{equation}
where $Q^{(c)}$ is the coarse-graining of $Q$, and $I_c^\mathrm{que}$ is the coarse-grained 
rate function. The argument in \cite[Section 8]{BGdH10} can be adapted to show, with the
help of Lemma~\ref{lem:spec_rel_ent_superadd}, that the supremum equals the rate function 
given in \eqref{Ique}, i.e., the coarse-grained relative entropies converge to the full relative 
entropies as $c\to\infty$. (Deuschel and Stroock~\cite[Lemma 4.4.15]{DeSt89} implies that 
the coarse-grained relative entropies are monotone in $c$.)

\medskip\noindent
{\bf 4.}
To obtain the quenched LDP, we must prove Eq.(3.1) and Eq.(4.1) in \cite{BGdH10} for the coarse-grained 
process. In Steps 1--3 this has already been achieved for $Q\in\Pinvfin((\widetilde
{\cE_c})^\Z)$ with $\Psi_Q\,\,{\ll}_{\rm{cond}}\,\,\nu^{(c)}$. Eq.(4.1) in \cite{BGdH10} 
trivially carries over when the latter restriction fails, but for Eq.(3.1) an additional
argument is needed. We must show that there exists a sequence $(\mathcal{O}_k(Q))_{k\in\N}$ 
of shrinking open neighborhoods of $Q$ such that 
\begin{equation}
\label{eq:shrinkopen}
\lim_{k\to\infty} \limsup_{N\to\infty} \frac{1}{N}\,\log 
\P^{(c)}\big(R^{(c)}_N \in \mathcal{O}_k(Q) \mid X^{(c)}\big) = -\infty,
\end{equation}
where $\P^{(c)}$ denotes the coarse-graining of $\P$. This can be done via an annealed 
estimate. Indeed, for $\nu^{(c)}$-a.a.\ $X^{(c)}$, 
\begin{equation}
\begin{aligned}
\limsup_{N\to\infty} \frac{1}{N}\,\log 
\P^{(c)}\big(R^{(c)}_N \in \mathcal{O}_k(Q) \mid X^{(c)}\big) 
&\leq \limsup_{N\to\infty} \frac{1}{N}\,\log 
\P^{(c)}\big(R^{(c)}_N \in \mathcal{O}_k(Q) \big)\\
&\leq - \inf_{Q' \in \overline{\mathcal{O}_k(Q)}} H(Q' \mid \P^{(c)}),
\end{aligned}
\end{equation}
where the last inequality follows from the annealed LDP. (This needs justification, since 
the annealed LDP was proved under condition (SV), which is not necessarily satisfied for
$\nu^{(c)}$. However, by Lemma \ref{lem:changingpastPolish2}, a decoupling inequality holds 
for a.a.\ pairs of coarse-grained pasts. Therefore there must be a regular conditional 
probability of $\nu^{(c)}$ satisfying Orey and Pelikan's ratio-mixing condition.) A 
sequence $(\mathcal{O}_k(Q))_{k\in\N}$ satisfying (\ref{eq:shrinkopen}) is easily obtained 
by letting $k\to\infty$ and using the lower semi-continuity of $Q' \mapsto H(Q' \mid 
\P^{(c)})$ together with the fact that $H(Q \mid P^{(c)}) \geq m_Q H(\Psi_Q \mid \nu^{(c)}) 
= \infty$ (see \cite[Eqs. (1.30--1.32)]{BGdH10}).
\end{proof}



\begin{thebibliography}{9}

\bibitem{B87}
H.\ Berbee,
\emph{Chains with infinite connections:  uniqueness and Markov representation},
Probab.\ Th.\ Rel.\ Fields 76 (1987) 243--253.

\bibitem{B08}
M.\ Birkner,
\emph{Conditional large deviations for a sequence of words},
Stoch.\ Proc.\ Appl.\ 118 (2008) 703--729.

\bibitem{BGdH10}
M.\ Birkner, A.\ Greven and F.\ den Hollander,
\emph{Quenched large deviation principle for words in a letter sequence},
Probab.\ Theory Related Fields 148 (2010) 403--456.

\bibitem{DeZe98} 
A.\ Dembo and O.\ Zeitouni, 
\emph{Large Deviations Techniques and Applications} (2nd Ed.), 
Springer, 1998.

\bibitem{DeSt89}
J.-D.\ Deuschel and D. W.\ Stroock,
\emph{Large Deviations},
Academic, Boston (1989).

\bibitem{DV83}
M.D.\ Donsker and S.R.S.\ Varadhan,
\emph{Asymptotic evaluation of certain Markov process expectations for large time.\ IV},
Comm.\ Pure Appl.\ Math.\ 36 (1983) 183--212.

\bibitem{DV85}
M.D.\ Donsker and S.R.S.\ Varadhan,
\emph{Large deviations for stationary Gaussian processes},
Comm.\ Math.\ Phys.\ 97 (1985) 187--210. 

\bibitem{H60}
P.R.\ Halmos,
\emph{Lectures on ergodic theory},
Chelsea Publishing Co.,\ New York, 1960.

\bibitem{dHS97}
F.\ den Hollander and J.E.\ Steif,
\emph{Mixing properties of the generalized $T,T^{-1}$-process}, 
J.\ d'Analyse Math.\ 72 (1997) 165--202.

\bibitem{JO}
A.\ Johansson and A.\ \"{O}berg,
\emph{Square summability of variations of g-functions and uniqueness of g-measures},
Math.\ Res.\ Lett.\ 10 (2003) 587--601.

\bibitem{JOP}
A.\ Johansson, A.\ \"{O}berg and M.\ Pollicott,
\emph{Countable state shifts and uniqueness of g-measures},
Amer.\ J.\ Math.\ 129 (2007) 1501--1511.

\bibitem{LD76}
F.\ Ledrappier,
\emph{Sur la condition de Bernoulli faible et ses applications},
in: \emph{Th\'eorie Ergodique} (Actes Journ\'ees Ergodiques, Rennes, 1973/1974),
Lecture Notes in Math.\ 532,\ Springer, Berlin, 1976, pp.\ 152--159.

\bibitem{O84}
S.\ Orey,
\emph{Large deviations in ergodic theory},
in: \emph{Seminar on Stochastic Processes, Evanston, Illinois, 1984},
Progr.\ Probab.\ Statist.\ 9 (1984) 195---249.

\bibitem{OP88}
S.\ Orey and S.\ Pelikan,
\emph{Large deviation principles for stationary processes},
Ann.\ Probab.\ 16 (1988) 1481--1495.

\bibitem{P67}
K.S.\ Parthasarathy,
\emph{Probability Measures on Metric Spaces}, Vol.\ 352, American Mathematical Society, 
1967.

\end{thebibliography}
\end{document}